\documentclass[11pt,a4paper,reqno]{amsart}

\usepackage{amsmath}
\usepackage{amssymb}
\usepackage{amsthm}
\usepackage{color}  
\usepackage{enumerate}
\usepackage{tikz-cd}
\usepackage[english]{babel}
\usepackage{ifthen}

\usepackage[pagebackref,
    ,pdfborder={0 0 0}%
    ,urlcolor=black,a4paper,hypertexnames=false]{hyperref}

\hypersetup{pdfauthor={Steffen Kionke, Carsten Feldkamp},%
            pdftitle={On upper bounds for the first $\ell^2$-Betti number}}

\title[Upper bounds for the first $\ell^2$-Betti number]{On upper bounds for the first $\ell^2$-Betti number} 

 \author[C. Feldkamp]{Carsten Feldkamp}
 \email{carsten.feldkamp@hhu.de}
 \author[S. Kionke]{Steffen Kionke}
 \address{Fakult\"{a}t f\"{u}r Mathematik und Informatik, FernUniversit\"{a}t in Hagen, 58084 Hagen, Germany}
 \email{steffen.kionke@fernuni-hagen.de}
 
 \thanks{Funded by the Deutsche Forschungsgemeinschaft (DFG, German Research Foundation) - 441848266}
 

\date{\today}

\theoremstyle{plain}
\newtheorem{theorem}{Theorem}[section]
\newtheorem{lemma}[theorem]{Lemma}
\newtheorem{corollary}[theorem]{Corollary}

\theoremstyle{definition}
\newtheorem{definition}[theorem]{Definition}

\newtheorem{remark}[theorem]{Remark}
\newtheorem*{remark*}{Remark}

\DeclareMathOperator{\Ch}{Ch}
\DeclareMathOperator{\Cay}{Cay}

\providecommand{\bbN}{\mathbb{N}}

\providecommand{\bbZ}{\mathbb{Z}}

\providecommand{\bbC}{\mathbb{C}}
\renewcommand{\epsilon}{\varepsilon}
\renewcommand{\phi}{\varphi}
\renewcommand{\Re}{\mathrm{Re}}

\begin{document}

\begin{abstract}
This article presents a method for proving upper bounds for the first $\ell^2$-Betti number of groups using only the geometry of the Cayley graph. As an application we prove that Burnside groups of large prime exponent have vanishing first $\ell^2$-Betti number. 

Our approach extends to generalizations of $\ell^2$-Betti numbers, that are defined using characters. We illustrate this flexibility by generalizing results of Thom-Peterson on q-normal subgroups to this setting.
\end{abstract}

\maketitle

Over the last 30 years the $\ell^2$-Betti numbers have become a major tool in the investigation of infinite groups.
The purpose of this article is to explore the first $\ell^2$-Betti number of groups using only the geometry of the Cayley graph. Our method is based on Pichot's observation \cite[Propositon 2]{Pichot06} that the first $\ell^2$-Betti number can be expressed with the \emph{rate of relations} in the Cayley graph. It follows from an elementary identity (see Lemma \ref{lem:projection-estimate}) that explicit cycles in the Cayley graph give rise to upper bounds for the first $\ell^2$-Betti number. Surprisingly, these elementary bounds can be used to prove new results.
\begin{theorem}\label{thm:torsion-bound-intro} Let $p$ be a prime and
let $G$ be a torsion group of exponent~$p$.
Then  $b_1^{(2)}(G) \leq 2p-2$.
\end{theorem}  
Using a theorem of Gaboriau this implies a vanishing result for the first $\ell^2$-Betti number of Burnside groups $B(m,p)$ of exponent $p$.
\begin{corollary}\label{corB(3,100)}
Let $p$ be a prime number. If $p$ is sufficiently large, then $b_1^{(2)}(B(m,p)) = 0$.
\end{corollary}
On the other hand, suppose that $b_1^{(2)}(B(m,p)) \neq 0$ for some prime $p$. Then Theorem \ref{thm:torsion-bound-intro} offers a simple solution of the restricted Burnside problem for $m$-generated $p$-groups using the multiplication formula for $\ell^2$-Betti numbers of finite index subgroups.

\medskip 

Our method can neatly be adapted to character-theoretic generalizations of the first $\ell^{2}$-Betti number.   We recall that every character $\psi$ (see \cite[Def.~2.5]{Kio2018})  of the group $G$, gives rise to a $\psi$-Betti number $b_1^\psi(G)$; see \cite{Kio2018} or Section~\ref{sec:bettinumbers}. The ordinary Betti numbers and the $\ell^2$-Betti numbers are special cases of this construction. However, it is difficult to calculate or bound $\psi$-Betti numbers under general assumptions of $\psi$.

 We extend Pichot's observation to the general setting and we use our method to generalize a result of Thom-Peterson \cite[Theorem 5.6]{ThomPeterson} to $\psi$-Betti numbers; see Corollary \ref{cor:q-psi-normal}.
 Even for $\ell^2$-Betti numbers our argument contains a new proof of their result. This provides a convenient way to bound (and sometimes calculate) $\psi$-Betti numbers in some generality. We illustrate this by proving a vanishing result for certain $\psi$-Betti numbers of right-angled groups; see Theorem \ref{thm:right-angled-vanish}.
 
\medskip
 
In Section~\ref{sec:bettinumbers} we discuss basic results on $\psi$-Betti numbers and we introduce our main method. In Section~\ref{sec:torsion} we apply it in the case of $p$-torsion groups. Section~\ref{sec:q-normal} is concerned with q-normality and presents applications to $\psi$-Betti numbers.  
 
\section{Betti numbers and the Cayley graph} \label{sec:bettinumbers}
The following simple result is essential for our approach.
\begin{lemma}\label{lem:projection-estimate}
	Let $\mathcal{H}$ be a Hilbert space and let $W \subseteq \mathcal{H}$ be a subspace.
	Let $P\colon \mathcal{H} \to \overline{W}$ denote the orthogonal projection onto the closure of $W$.
Then for all $v \in \mathcal{H}$
\[
	\langle Pv, v \rangle = \sup_{w \in W} \frac{|\langle w,v \rangle|^2}{\langle w, w \rangle}
\]
where the supremum is taken over all non-zero elements of $W$ (and is defined to be $0$ if $W = 0$).
\end{lemma}
\begin{proof}
For $v = 0$ the assertion is obvious. We may assume that $\Vert v \Vert = 1$.
For all $w \in W$, we note that
\begin{align*}
 |\langle w , v \rangle|^2  = 
|\langle Pw , v \rangle|^2 = |\langle w , Pv \rangle|^2 \stackrel{C.S.}{\leq} \Vert w \Vert^2 \Vert Pv \Vert^2 
= \langle Pv, v \rangle
 \Vert w \Vert^2.
\end{align*}
If $w \neq 0$ we obtain
\[
	\frac{|\langle w , v \rangle|^2}{\langle w, w \rangle} \leq \langle Pv, v \rangle.
\]
In particular, the proof is complete if $Pv =0$. 

For the converse we assume $Pv \neq 0$. Let $\epsilon \in (0,1)$. Since $W$ is dense in $\overline{W}$, there is $w \in W$ with
$\Vert Pv-w \Vert < \epsilon \Vert Pv \Vert^2$ and we deduce
\[
	|\langle w, v \rangle| \geq  \langle Pv, v \rangle - |\langle w- Pv, v \rangle| \stackrel{C.S.}{\geq}\langle Pv, v \rangle - \Vert w- Pv \Vert  \geq (1-\epsilon) \langle Pv, v \rangle.
\]
In addition, we note that $\Vert w \Vert = \Vert w -Pv + Pv \Vert \leq (1+\epsilon) \Vert Pv \Vert$ and so
\[
\frac{|\langle w , v \rangle|^2}{\langle w, w \rangle} \geq \frac{(1-\epsilon)^2}{(1+\epsilon)^2} \langle Pv, v \rangle
\]
The assertion follows as $\epsilon$ can be arbitrarily close to $0$.
\end{proof}

Let $G$ be a group. A \emph{character} of $G$ is a function $\psi\colon G \to \bbC$ of positive type, which is constant on conjugacy classes of $G$ and satisfies $\psi(1_G) = 1$; see \cite[Def.~2.5]{Kio2018}. Let $\Ch(G)$ denote the space of all characters of $G$. Every character $\psi \in  \Ch(G)$ gives rise to a semi-definite $G$-invariant inner product $\langle g, h \rangle_\psi = \psi(h^{-1}g)$ on the group ring $\bbC[G]$. Passing to the completion provides us with a tracial Hilbert $G$-bimodule $\ell^\psi(G)$; see \cite[Def.~2.1]{Kio2018}. Using the GNS construction, this provides a tracial von Neumann algebra and a notion of dimension, which can be used to define the $\psi$-Betti numbers $b_k^\psi(G)$ of $G$, provided that $G$ satisfies suitable finiteness properties. Specifically $b_0^\psi$ is defined for all groups and $b_1^\psi$ is defined for all finitely generated groups.

For the regular character $\delta$ with $\delta(g) = 0$ for all $g \neq 1_G$, one has $\ell^\delta(G) = \ell^2(G)$ and one obtains the famous $\ell^2$-Betti numbers $b_k^{(2)}(G)$. The constant character (i.e., $\psi(g) = 1$ for all $g$) gives rise to the ordinary rational Betti numbers of $G$ since $\ell^\psi(G) \cong \bbC$.

\begin{definition}
Let $G$ be a group and let $\psi \in \Ch(G)$. A subgroup $K \leq G$ is \emph{$\psi$-regular}, if $\psi|_K$ is the regular character on $K$, i.e.~$\psi(k) = 0$ for all $k \in K \backslash \{1\}$. 
\end{definition}

Here we are mainly interested in the first Betti numbers $b_1^\psi(G)$.
It will however be useful and instructive to initially consider the $0$-th Betti number. 
Let $J_G$ denote the augmentation ideal in $\bbC[G]$, i.e.~the set of elements $w = \sum_{g\in G} w_g g$ which satisfy $\sum_{g \in G} w_g = 0$. 
\begin{lemma}\label{lem:b0}
Let $G$ be a group and let $\psi \in \Ch(G)$ be a character.
\begin{enumerate}[ (a)]
\item\label{it:b0-formula} $b_0^\psi(G) = 1-  \sup_{w \in J_G} \frac{\langle w, 1 \rangle_\psi^2}{\langle w, w \rangle_\psi}$
where the supremum is taken over all non-zero elements of $J_G$.

\item\label{it:b0-lim} If $G = \bigcup_{i\in I} G_i$ is a directed union of subgroups $G_i$, then 
$\lim_{i \in I} b^\psi_0(G_i) = b_0^\psi(G)$.

\item\label{it:b0-regular} $b_0^\psi(G) \leq \frac{1}{|K|}$ 
for every $\psi$-regular
subgroup $K \leq G$.
\end{enumerate}
\end{lemma}
\begin{remark}
It is well-known that $b_0^{(2)}(G) = \frac{1}{|G|}$; see \cite[Thm.~1.35 (8)]{Luck}.
\end{remark}
\begin{proof}
Let $S$ be a generating set for $G$. 
We consider the initial segment of the associated free resolution of $\bbC$:
\[
     \bbC[G]^{S} \stackrel{\partial_1}{\longrightarrow} \bbC[G] \longrightarrow \bbC.
\]
The image of $\partial_1$ is the augmentation ideal.
We take the tensor product with $\ell^\psi(G)$ and deduce that
\[
	b_0^\psi(G) = 1 - \dim_\psi(\overline{J_G}).
\]
where $\overline{J_G}$ denotes the closure of the image of the augmentation ideal in $\ell^\psi(G)$.
Let $P \colon \ell^\psi(G) \to \overline{J_G}$ denote the orthogonal projection. By definition
\[
	\dim_\psi(\overline{J_G}) = \langle P(1), 1 \rangle_\psi
\]
and assertion \eqref{it:b0-formula} follows from Lemma \ref{lem:projection-estimate}. Let $G = \bigcup G_i$ be a direct union of subgroups, then $J_G = \bigcup J_{G_i}$ and \eqref{it:b0-lim} follows immediately from \eqref{it:b0-formula}.

Let $K \leq G$ be a $\psi$-regular subgroup. Let $T \subseteq K \setminus \{1\}$ be a finite subset. Then
\[
	w = |T|\cdot 1_G - \sum_{k \in T} k \in J_G.
\]
 Since $K$ is $\psi$-regular, the elements of $K$ are orthonormal and we deduce
\[
	\frac{|\langle w ,1 \rangle_\psi|^2}{\langle w, w \rangle_\psi} = \frac{|T|^2}{|T|^2 + |T|} = \frac{|T|}{|T|+1}.
\]
Now \eqref{it:b0-formula} implies $b_0^\psi(K) \leq 1 - \frac{|T|}{|T|+1} = \frac{1}{|T|+1}$.
Statement \eqref{it:b0-regular} follows by taking $T = K \setminus \{1\}$ if $K$ is finite respectively letting $|T|$ tend to $\infty$ otherwise.
\end{proof}

We would like to apply the same ideas to the first $\psi$-Betti number~$b_1^\psi(G)$. However, up to now we only have a definition of $b_1^\psi(G)$ for all finitely generated groups $G$. We also require a definition for groups which are not finitely generated. This could be done using  L\"uck's generalized dimension function (discussed in \cite[\S6.1, 6.2]{Luck}), but this is not convenient for our purposes and for simplicity we work with the following variation.

\begin{definition}
Let $G$ be a group and let $\psi \in \Ch(G)$. Then
\[
	\bar{b}_1^\psi(G) := \liminf_{H \leq G} b_1^\psi(H)
\]
where the limit is taken over the directed system of all finitely generated subgroups $H \leq G$.
\end{definition}
\begin{remark}\label{rem:bar}
For a finitely generated group $\bar{b}_1^\psi(G) = b_1^\psi(G)$.
In general however,
$\bar{b}_1^\psi(G)$ can be strictly larger than the properly defined value of the first $\psi$-Betti number. It is easy to see this for the ordinary Betti numbers. For instance, it follows from the methods developed in \cite{Howie82} that $\langle (x_i)_{i \in \bbZ} \mid x_ix_{i+1}x_i^{-1} = x_{i+1} ^2\rangle$, is a perfect and locally indicable group, i.e., the ordinary rational Betti number of every finitely generated subgroup is $\geq 1$.

For the classical $\ell^2$-Betti number the inequality $b_1^{(2)}(G) \leq \bar{b}_1^{(2)}(G)$ follows from the argument given in the proof of \cite[Theorem 7.2 (3)]{Luck}.
\end{remark}
For later reference we state the following observation.
\begin{lemma}\label{lem:liminf}
Let $G$ be a group and let $\psi \in \Ch(G)$.
If $G = \bigcup_{i\in I} G_i$ is a directed union of subgroups $G_i$, then
\[
	\bar{b}_1^\psi(G) \leq \liminf_{i\in I} \bar{b}_1^\psi(G_i)
\]
\end{lemma}
\begin{proof}
Let $\epsilon > 0$. There is a finitely generated subgroup $H_0 \leq G$ such that $b_1^\psi(H) \geq \bar{b}_1^\psi(G) - \epsilon$ for all finitely generated subgroups $H$ that contain $H_0$. Since $H_0$ is finitely generated, there is $i \in I$ such that $H_0 \subseteq G_i$.  Thus for all $j \geq i$ we have 
$\bar{b}_1^{\psi}(G_j) \geq \bar{b}_1^\psi(G)-\epsilon$.
\end{proof}

Assume that $G$ is finitely generated and that $S$ is a finite generating set.
The Cayley graph $\Cay(G,S)$ is the directed graph with vertex set $G$ and edges 
\[
	E_{G,S} = \{ (g,gs) \mid g \in G, s \in S\}.
\]
The edge $(1_G, s)$ will be denoted by $\bar{s}$.
The Cayley graph is equipped with a left action of $G$.
Let $\bbC[E_{G,S}]$ be the vector space with basis $E_{G,S}$ and let $\partial\colon \bbC[E_{G,S}] \to \bbC[G]$ denote the boundary map.
A \emph{finite cycle} in $\Cay(G,S)$ is an element $z \in \bbC[E_{G,S}]$ with $\partial(z) = 0$. Let $Z_{G,S}$ denote the space of finite cycles.
If $\psi\in \Ch(G)$ is a character, then the semi-definite inner product $\langle \cdot, \cdot\rangle_\psi$ extends to a $G$-invariant semi-definite inner product on $\bbC[E_{G,S}]$ such that the edges $\{\bar{s} \mid s \in S\}$ are orthonormal; this means
\[
	\langle (g,gs), (h,ht) \rangle_\psi = \begin{cases} 0 & \text{ if } s \neq t\\
	\psi(h^{-1}g) & \text{ if } s = t\end{cases}
\]
 The following extends \cite[Prop.~2]{Pichot06} and is our main tool.

\begin{lemma}\label{lem:main-lemma}
Let $G$ be a group and let $S$ be a finite generating set. Let $\psi \in \Ch(G)$ be a character.
Then
\[
	b_1^\psi(G) = |S| - 1 + b_0^{\psi}(G) - \sum_{s \in S} \sup_{z \in Z_{G,S}} \frac{|\langle z, \bar{s} \rangle_\psi|^2}{\langle z, z \rangle_\psi}
\]
where the suprema are taken over all non-zero elements of $Z_{G,S}$.
\end{lemma}
\begin{proof}
The finite generating set $S$ provides us with a presentation $G \cong F/R$ of $G$ where $F$ is the free group over $S$ and $R$ is the subgroup of relations.
We consider the initial segment of the associated free resolution of $\bbC$:
\[
	\bbC[G]^{R}\stackrel{\partial_2}{\longrightarrow} \bbC[G]^{S} \stackrel{\partial_1}{\longrightarrow} \bbC[G] \longrightarrow \bbC.
\]
Tensoring with $\ell^\psi(G)$ gives
\[
	\ell^\psi(G)^{R}\stackrel{\partial_2}{\longrightarrow} \ell^\psi(G)^{S} \stackrel{\partial_1}{\longrightarrow} \ell^\psi(G).
\]
The middle term is naturally isomorphic to the completion of $\bbC[E_{G,S}]$ with respect to $\langle \cdot, \cdot \rangle_\psi$. The image of $\partial_2$ is the closure of $Z_{G,S}$. The $\psi$-dimension of the closure of the image of $\partial_1$ is $1 - b_0^\psi(G)$. We deduce that
\[
	b_1^\psi(G) = |S| - (1-b_0^\psi(G)) - \dim_\psi(\overline{Z_{G,S}}).
\]
Let $P \colon \ell^\psi(G)^S \to \overline{Z_{G,S}}$ denote the orthogonal projection. By definition
\[
	\dim_\psi(\overline{Z_{G,S}}) = \sum_{s \in S} \langle P\bar{s}, \bar{s} \rangle_\psi
\]
and the result follows from Lemma \ref{lem:projection-estimate}
\end{proof}
\begin{remark}\label{rem:b0}
It seems surprising that the value on the right hand side is independent from the chosen set of generators. This is a consequence of the homotopy invariance of the $\psi$-Betti numbers, which can be proven using the standard argument; e.g.~\cite[Thm. 3.18]{Kammeyer} or~\cite{Luck}.  
\end{remark}

\section{Torsion groups}\label{sec:torsion}
In view of Remark \ref{rem:bar}, the following result implies Theorem~\ref{thm:torsion-bound-intro}.
\begin{theorem}\label{thm:torsion-bound}
Let $p$ be a prime.
Let $G$ be a torsion group of exponent $p$.
Then  $\bar{b}_1^{(2)}(G) \leq 2p-2$.
\end{theorem}
\begin{proof}
We may assume that $G$ is infinite (and $b_0^{(2)} = 0$), otherwise $b_1^{(2)}(G) = 0$ and there is nothing to show.
By the definition of  $\bar{b}_1^{(2)}(G)$ (see Remark \ref{rem:bar}), we may assume that $G$ is finitely generated.

We choose a minimal generating set $S$ of $G$ and denote the number of elements by $N = |S|$. Since all elements of $G$ have prime order, all pairwise distinct elements $a,b,c \in S$ satisfy
\begin{equation}\label{condition:cyclic}
	\langle ac \rangle \cap \langle ab \rangle = \{1\}
\end{equation}
Suppose for a contradiction that there are three distinct elements $a,b,c \in S$ with
$\langle ac \rangle \cap \langle ab \rangle \neq \{1\}$, then these cyclic groups of prime order coincide and 
\[
	ac = (ab)^k
\]
for some $k \in \bbN$, i.e., $c \in \langle a,b \rangle$ which contradicts the minimality of $S$.

For all $a \in S$ we have $N-1$ relations
\[
	(ab)^p
\] 
of length $2p$
for all $b \neq a$ in $S$.
By condition \eqref{condition:cyclic}, the only common edge in the Cayley graph is the first edge $\bar{a}$ from $1$ to $a$.
Summing up these cycles, we obtain a cycle $z_a$ in $\mathrm{Cay}(G,S)$ with
\[
	\frac{\langle z_a,\bar{a} \rangle^2}{\langle z_a, z_a \rangle} = \frac{(N-1)^2}{(N-1)^2+(N-1)(2p-1)} = \frac{1}{1+\frac{2p-1}{N-1}}.
\]
We deduce from Lemma \ref{lem:main-lemma} that
\begin{eqnarray*}
b_1^{(2)}(G)  &\leq& N-1 - \sum_{a \in S} \frac{\langle z_a,\bar{a} \rangle^2}{\langle z_a, z_a \rangle} \ \ = \ \  N-1 - \frac{N}{1+\frac{2p-1}{N-1}} \\
	&=& \frac{2p-2}{1+\frac{2p-1}{N-1}} \ \ \leqslant \ \ 2p-2.
\end{eqnarray*}
\end{proof}
\begin{theorem}\label{thm:vanishing-torsion}
Let $p$ be a prime number and let $G$ be a countable torsion group of exponent $p$.
If $G$ has an infinite normal subgroup $N$ of infinite index, then 
\[
	b_1^{(2)}(G) = 0.
\]
\end{theorem}
\begin{proof}
By Theorem \ref{thm:torsion-bound} and Remark \ref{rem:bar} we have $b_1^{(2)}(N) \leq \bar{b}_1^{(2)}(N) \leq 2p-2$. By Gaboriau's Theorem~\cite[Thm.~6.8]{Gaboriau2002}, this implies $b_1^{(2)}(G) = 0$.
\end{proof}

\begin{proof}[Proof of Corollary \ref{corB(3,100)}]
Recall that $B(m,p)$ denotes the Burnside group of exponent $p$ and rank $m$. 
Since $B(1,p)$ is finite, we have $b_1^{(2)}(B(1,p)) = 0$.

Assume $m \geq 2$.
For sufficiently large $p$, the main result of \cite{Ivanov2003} implies that $B(m,p)$ contains a $Q$-subgroup $H$ which is isomorphic to $B(\infty,p)$. 
A $Q$-subgroup has the property that the normal closure $\langle K \rangle^{B(m,p)}$ in $B(m,p)$ of any normal subgroup $K \trianglelefteq H$ 
intersects $H$ exactly in $K$.

Take a projection from $B(\infty,p)$ onto $B(\infty,p)$ with an infinite kernel $K$. Then the normal closure $\langle K \rangle^{B(m,p)}$ is an infinite normal subgroup of $B(m,p)$ of infinite index. Now Theorem \ref{thm:vanishing-torsion} implies the result.
\end{proof}

\begin{remark}\label{remB(3,100)}
\begin{enumerate}
\item Ivanov \cite{Ivanov2003} quantifies sufficiently large as $p > 10^{78}$.

\item  One can also deduce $b_1^{(2)}(B(m,p)) = 0$ for $m \geq 3$ under the assumption that $B(2,p)$ is infinite\footnote{According to Adian \cite{Adian2015} the Burnside groups $B(2,p)$ are infinite for all $p > 100$.}
using the normal subgroup $N = \ker( B(m,p) \to B(m-1,p) )$.

Indeed, let $x_1,x_2,\dots,x_m$ be a free generating set of $B(m,p)$ such that $N$ is the normal closure of $x_1$ in $B(m,p)$. Since $\langle x_1, x_2 \rangle \subseteq N\langle x_2 \rangle$ and $\langle x_1,x_2 \rangle \cong B(p,2)$ is infinite, we deduce that $N$ is infinite. Moreover, $N$ has infinite index, since  $B(m,p)/N \cong B(m-1,p)$.
\item We expect that $b_1^{(2)}(B(m,p)) = 0$ for all $p, m$. On the other hand, if $b_1^{(2)}(B(m,p)) > 0$ holds for some $m$ and $p$, then this offers a simple solution to the restricted Burnside problem for $m$-generated $p$-groups. More precisely,
every finite index normal subgroup $N \trianglelefteq B(m,p)$ satisfies
\[
	|B(m,p):N| \cdot  b_1^{(2)}(B(m,p)) = b_1^{(2)}(N) \leq 2p-2
\]
by Theorem \ref{thm:torsion-bound} and this inequality imposes an upper bound on the index of $N$.
\end{enumerate} 
\end{remark}

\section{$q$-normality and applications}\label{sec:q-normal}

\begin{lemma} \label{lemHa}
Let $G=\langle H,a\rangle$ be a group and let $\psi\in\Ch(G)$ be a character. Assume that $aHa^{-1} \cap H$ contains a $\psi$-regular subgroup of order $n \in \mathbb{N} \cup \{\infty\}$. Then
\[
\bar{b}_{1}^{\psi}(G) -  b_0^\psi(G) \leqslant \bar{b}_{1}^{\psi}(H) + \frac{3+ 2\Re(\psi(a))}{n+2+2\Re(\psi(a))} 
\]
In particular
 \[
\bar{b}_{1}^{\psi}(G) \leqslant \bar{b}_{1}^{\psi}(H).\]
for $n = \infty$.
\end{lemma}
\begin{proof}
Without loss of generality we assume that $a \not\in H$.
We denote the $\psi$-regular subgroup of order $n$ in $H \cap aHa^{-1}$ by $K$.
For a finite subset $S \subseteq H$, we define $H_S = \langle S \rangle$.
Since $K$ is $\psi$-regular, we obtain $b_0^\psi(H) \leqslant b_0^\psi(G) \leqslant \frac{1}{n}$ and $b_0^\psi(H_S) \leq |K\cap H_S|^{-1}$ by Remark~\ref{rem:b0}.

For $n = \infty$, let $S \subseteq H$ be any finite subset and denote by $h_1,h_2, \dots, h_k$ the pairwise distinct elements of $S \cap aSa^{-1} \cap K$. If $n< \infty$, we choose $S \subseteq H$ such that
 $$S \cap aSa^{-1} \ = 
\ K \setminus\{1\} \ = \ \{h_1,h_2, \dots, h_k\}.$$
In both situations we define $S' = S \cup \{a\}$ and $G' = \langle S' \rangle$. 
Lemma \ref{lem:main-lemma} implies
\begin{eqnarray}
b_{1}^{\psi}(G') - b_0^\psi(G')  &=& |S'| - 1 - \sum_{s \in S'} \sup_{z \in Z_{G',S'}} \frac{|\langle z,\bar{s}\rangle_\psi|^2}{\langle z,z\rangle_\psi}\nonumber\\
&\leq& | S | + 1 - 1 - \sum_{s \in S} \sup_{z \in Z_{H_S,S}} \frac{|\langle z,\bar{s}\rangle_\psi|^2}{\langle z,z\rangle_\psi} -  \sup_{z \in Z_{G',S'}} \frac{|\langle z,\bar{a}\rangle_\psi|^2}{\langle z,z\rangle_\psi}\nonumber\\
&\leq& b_{1}^{(2)}(H_S)  + 1 - \sup_{z \in Z_{G,S'}} \frac{|\langle z,\bar{a}\rangle_\psi|^2}{\langle z,z\rangle_\psi}. \label{inequ1}
\end{eqnarray} 
To obtain a lower bound for $\sup_{z \in Z_{G,S'}} \frac{|\langle z,\bar{a}\rangle_\psi|^2}{\langle z,z\rangle_\psi}$, we consider the Cayley graph $\Cay(G',S')$ of $G'$ and exhibit a suitable cycle $z$. Each relation $ah_ia^{-1}(ah_ia^{-1})^{-1}$ provides a cycle $z_i$ of length $4$ in $\Cay(G',S')$, i.e.,
\[
	z_i = \underbrace{(1,a)}_{=\bar{a}} + (a,ah_i) - (ah_ia^{-1},ah_i) - (1,ah_ia^{-1}).
\] 
Note that the cycles $z_i$ touch exactly four vertices, since $h_i \neq 1$ and $a \not\in H$. 
In additon, the cycles $z_i$ have no common edges, except for $\bar{a}$. We define $z = \sum_{i=1}^k z_i$ to be the sum of these cycles. 
Since $\psi(ah_ia^{-1}) = \psi(h_i) = 0$ holds for all $i\leq k$, we deduce
\[
	\langle z,\bar{a} \rangle_\psi =  k \langle \bar{a}, \bar{a} \rangle_\psi - \sum_{i=1}^k \langle (ah_ia^{-1},ah_i), \bar{a} \rangle_\psi =
	k -  \sum_{i=1}^k \psi(ah_ia^{-1}) = k
\]
We note further (using again that $K$ is $\psi$-regular) that for given $i,j \leq k$ the edges $\neq \bar{a}$ in $z_i, z_j$ are orthogonal unless $h_i = ah_ja^{-1}$ or $ah_ia^{-1} = h_j$. Each of these cases occurs at most once for every $i$ and then $\langle (a,ah_i), (1,ah_ja^{-1}) \rangle_\psi = \psi(a)$ and $\langle (1,ah_ia^{-1}), (a,ah_j) \rangle_\psi = \overline{\psi(a)}$ respectively. We deduce 
\begin{align}
\langle z, z\rangle_\psi &= \sum_{i,j} \langle z_i,z_j \rangle_\psi 
\leq  \underbrace{4k}_{i=j} +  \underbrace{k^2 - k + k (\psi(a) + \overline{\psi(a)})}_{i\neq j} \nonumber \\ &\leq k^2 + (3 +2 \Re(\psi(a))) k \label{eq:bound}
\end{align}
and conclude
\[
	\frac{\langle z,\bar{a}\rangle_\psi^2}{\langle z,z\rangle_\psi} \geq \frac{1}{1 + \frac{3+2\Re(\psi(a))}{k}}.
\]
Finally, we use this cycle in combination with inequality \eqref{inequ1} to obtain
\begin{eqnarray*}
b_{1}^{\psi}(G') - b_0^\psi(G') &\leqslant&  b_{1}^{(2)}(H_S) + 1 - \sup_{z \in Z_{G,S'}} \frac{|\langle z,\bar{a}\rangle_\psi|^2}{\langle z,z\rangle_\psi}\\
&\leqslant& b_{1}^{(2)}(H_S)  + 1 -  \frac{1}{1 + \frac{3 +2 \Re(\psi(a))}{k}}\\
&=& b_{1}^{(2)}(H_S) + \frac{3 +2 \Re(\psi(a))}{k+3 +2 \Re(\psi(a))}  \ \ \ \underset{k \rightarrow \infty}{\longrightarrow} \ b_{1}^{(2)}(H_S)
\end{eqnarray*}
For $n=\infty$ we can make $k$ arbitrary large. In the case $n<\infty$ we have $k=n-1$ by construction.
We note that every finitely generated subgroup of $G$ is contained in a group of the form $G'$.
The result follows from Lemma \ref{lem:liminf} and Lemma \ref{lem:b0} \eqref{it:b0-lim}.

If $n = \infty$, then $G$ contains an infinite $\psi$-regular subgroup and $b_0^\psi(G) = 0$ by Lemma \ref{lem:b0} \eqref{it:b0-regular}.
\end{proof}
In the spirit of Popa \cite{Popa} and Thom-Peterson \cite{ThomPeterson} we introduce the following notion.
\begin{definition}
Let $G$ be a group and let $\psi\in \Ch(G)$. A subgroup $H \leq G$ is \emph{q-$\psi$-normal}, if there is a set $A \subseteq G$ such that $G = \langle H \cup A \rangle$ and $H \cap aHa^{-1}$ contains an infinite $\psi$-regular subgroup for all $a \in A$.

A subgroup $H \leq G$ is \emph{weakly} q-$\psi$-normal, if there is an ordinal number $\alpha$ and an increasing chain of subgroups $H_0 = H$ to $H_\alpha = G$ such that 
$\bigcup_{\beta < \gamma} H_\beta$ is q-$\psi$-normal in $H_\gamma$ for all $\gamma \leq \alpha$.
\end{definition}
Based on Lemma \ref{lemHa} we obtain the following analog of~\cite[Theorem~5.6]{ThomPeterson}.
\begin{corollary}\label{cor:q-psi-normal}
Let $G$ be a group and let $\psi \in \Ch(G)$. If $H \leq G$ is a weakly q-$\psi$-normal subgroup, then 
\[
	\bar{b}_1^\psi(G) \leq \bar{b}_1^{\psi}(H).
\]
\end{corollary}
\begin{proof} 
Assume that $H$ is q-$\psi$-normal.
Then $G = \langle H \cup A \rangle$ and $H \cap aHa^{-1}$ contains an infinite $\psi$-regular subgroup for all $a \in A$. 
If $A$ is finite, then the assertion follows inductively from Lemma \ref{lemHa}. Assume that $A$ is infinite. For every finite subset $B \subseteq A$, we define $G_B = \langle H \cup B \rangle$. Then $G = \bigcup_{B \subseteq A} G_B$ and Lemma \ref{lem:liminf} implies
\[
	\bar{b}_1^{\psi}(G) \leq \liminf_{B \subseteq A} \bar{b}_1^\psi(G_B) \leq \bar{b}_1^\psi(H).
\]
The result for weakly q-$\psi$-normal subgroups follows by transfinite induction using Lemma~\ref{lem:liminf}.
\end{proof}

\begin{corollary}
Let $G$ be a group and let $\psi \in \Ch(G)$.
\begin{itemize}
\item[(1)] If $G$ is an HNN-extension of $H$ with associated subgroups $A, B$ and $A$ contains an infinite $\psi$-regular subgroup, then we have $\bar{b}_{1}^{\psi}(G) \leq \bar{b}_{1}^{\psi}(H)$.
\item[(2)] If $G= A \ast_{C} B$ is an amalgamated product such that $C$ is q-$\psi$-normal in $B$. Then we have $\overline{b}_{1}^{\psi}(G) \leq \bar{b}_{1}^{\psi}(A)$.
\item[(3)] If $G$ contains an infinite normal amenable $\psi$-regular subgroup, then $b^\psi_1(G) = 0$.
\end{itemize}
\end{corollary}
\begin{proof}
(1): follows immediately from Lemma \ref{lemHa}. 

(2): The assumptions imply that $A$ is q-$\psi$-normal in $G = A*_C B$ and the assertion follows from Corollary~\ref{cor:q-psi-normal}.

(3): The infinite normal amenable subgroup $N$ is q-$\psi$-normal in $G$ and $b_1^\psi(N) = b_1^{(2)}(N) = 0$ by~\cite[Thm.~0.2]{CheegerGromov}.
\end{proof}

We illustrate the helpfulness of q-normality with an application to right-angled groups. This notion was put forward in \cite[Definition 1]{AGN17}.
\begin{definition} [right-angled groups]
A group $G$ is \emph{right-angled}, if it is the quotient
of a right-angled Artin group $A_{\Gamma}$ with a finite connected graph $\Gamma = (\mathcal{I},\mathcal{E})$ such that the image of every generator 
$\sigma_{i}$ $(i \in \mathcal{I})$ has infinite order in $G$.

The image of the generating set of $A_{\Gamma}$ will be called a \emph{right-angled set of generators}.
\end{definition}

\begin{theorem}\label{thm:right-angled-vanish}
Let $G$ be a right-angled group and let $\mathcal{S} = \{s_{i} \mid i \in \mathcal{I}\}$ be a right-angled set of generators. If $\psi \in \Ch(G)$ is such that the cyclic subgroup $\langle s_i \rangle$ is $\psi$-regular for every $i\in \mathcal{I}$. Then we have $b_{1}^{\psi}(G)=0$.
\end{theorem}

\begin{proof}
Our proof will be by induction over the number $n = |\mathcal{I}| \in \mathbb{N}$ of generators. For the base of induction we note that $b_1^{(2)}(\mathbb{Z}) = 0$. We assume for the induction step w.\,l.\,o.\,g. $\mathcal{S}=\{s_{1},s_{2},\dots,s_{n},s_{n+1}\}$ such that $s_{n}$ commutes with $s_{n+1}$ and such that $G' = \langle s_1 \dots, s_n \rangle$ is a right-angled group with $b_{1}^{(2)}(G')=0$. We claim that $G'$ is q-$\psi$-normal. Indeed, set $H= \langle s_{n} \rangle$ and $a=s_{n+1}$,  $a^{-1}G'a \cap G' \supseteq H$ and $H$ is $\psi$-regular by assumption. Now the result follows from Lemma~\ref{lemHa}.
\end{proof}
Using this calculation and the approximation methods from \cite{Kio2018} 
one can control the growth of Betti numbers in right-angled Artin groups with respect to normal chains with non-trivial intersection.
\begin{corollary}
Let $A_\Gamma$ be a right-angled Artin group for a finite connected graph $\Gamma$ with generating set $\{\sigma_i \mid i \in \mathcal{I}\}$. Let $N_1 \trianglerighteq N_2 \trianglerighteq \dots$ be a descending chain of finite index normal subgroups in $A_\Gamma$. If the order $\mathrm{ord}_{A_\Gamma/N_n}(\sigma_i)$ in the finite factors $A_\Gamma/N_n$ is unbounded for each generator $\sigma_i$, then
\[
	\lim_{n \to \infty} \frac{b_1(N_n)}{|A_\Gamma:N_n|} = 0.
\]
\end{corollary}
\begin{proof}
Let $\psi_n$ be the character of the permutation action of $A_\Gamma$ on $A_\Gamma/N_n$. Since the sequence of normal subgroups is descending, the sequence $\psi_n$ converges 
in $\Ch(A_\Gamma)$ to a character $\psi$. 
Since each character $\psi_n$ is sofic and $A_\Gamma$ is finitely presented, it follows from \cite[Theorem 3.5]{Kio2018} that
\[
	b_1^\psi(A_\Gamma) = \lim_{n \to \infty} b^{\psi_n}_1(A_\Gamma) = \lim_{n \to \infty} \frac{b_1(N_n)}{|A_\Gamma:N_n|}
\]
If the order $\mathrm{ord}_{A_\Gamma/N_n}(\sigma_i)$ tends to infinity, $\psi_n(\sigma_i^k)$ vanishes for all $k \neq 0 $ and all large $n$, i.e., $\langle \sigma_i \rangle$ is a $\psi$-regular subgroup.
Theorem \ref{thm:right-angled-vanish} implies that $b_1^\psi(A_\Gamma) = 0$ and this completes the proof.
\end{proof}

\bibliography{Literatur}
\bibliographystyle{abbrv}

\end{document}